\documentclass{article}
\usepackage[T2A]{fontenc}
\usepackage[russian]{babel}
\usepackage[utf8]{inputenc}
\usepackage[tbtags]{amsmath}
\usepackage{amsfonts,amssymb,mathrsfs,amscd,comment, amsthm, graphicx}

\def\R{{\mathbb R}}

\def\Z{\mathbb Z}

\newcommand{\Id}{\operatorname{Id}\nolimits}

\def\eps{\varepsilon}
\def\lam{\lambda}
\def\f{\varphi}

\def\cut{\operatorname{cut}}
\def\conj{\operatorname{conj}}
\def\Cut{\operatorname{Cut}}
\def\Conj{\operatorname{Conj^1}}
\def\Conjo{\operatorname{Conj^1_0}}
\def\Conje{\operatorname{Conj^1_{\eps}}}
\def\tconj{t_{\conj}^1}
\def\tcut{t_{\cut}}
\def\const{\operatorname{const}}
\def\Exp{\operatorname{Exp}}

\renewcommand{\leq}{\leqslant}
\renewcommand{\geq}{\geqslant}

\newcommand{\pder}[2]{\frac{\partial \, #1}{\partial \, #2} }
\newcommand{\restr}[2]{\left. #1 \right|_{#2}}
\newcommand{\map}[3]{#1 \, : \, #2 \to #3}
\newcommand{\eq}[1]{$(\protect\ref{#1})$}
\newcommand{\be}[1]{\begin{equation}\label{#1}}
\newcommand{\ee}{\end{equation}}
\def\then{\quad\Rightarrow\quad}

\theoremstyle{plain}
\newtheorem{theorem}{Теорема}
\newtheorem{proposition}{Предложение}
\newtheorem{corollary}{Следствие}

\theoremstyle{remark}
\newtheorem{remark}{Замечание}

\newcommand{\onefiglabelsizen}[4]
{
\begin{figure}[htbp]
\begin{center}
\includegraphics[height=#4cm]{#1}
\\
\parbox[t]{0.7\textwidth}{\caption{#2}\label{#3}}
\end{center}
\end{figure}
}

\newcommand{\twofiglabelsize}[8]
{
\begin{figure}[htbp]
\includegraphics[height=#4\textwidth]{#1}
\hfill
\includegraphics[height=#8\textwidth]{#5}
\\
\hfill
\parbox[t]{0.45\textwidth}{\caption{#2}\label{#3}}
\hfill
\parbox[t]{0.45\textwidth}{\caption{#6}\label{#7}}
\hfill
\end{figure}
}

\graphicspath{{figures/}}

\begin{document}

\title{
Семейство римановых задач на группе Гейзенберга
}

\author{Ю.Л.~Сачков}
\date{}

\maketitle

\begin{abstract}
Рассматривается однопараметрическое семейство римановых задач на группе Гейзенберга, сходящееся к субримановой задаче на этой группе при стремлении параметра к предельному значению. Для семейства построен оптимальный синтез, описаны сферы, множество разреза и первая каустика. При стремлении параметра к предельному значению показано, как экспоненциальное отображение, первая каустика, множество разреза и сферы стремятся к соответствующим объектам для субримановой задачи.
\end{abstract}

\tableofcontents

\section{Введение}

Как известно, субриманову структуру можно получить как предел подходящего семейства римановых структур, когда движение допустимых кривых сильно штрафуется в направлениях, трансверсальных распределению субримановой структуры. Цель данной работы --- рассмотреть этот предельный переход в случае левоинвариантных структур на группе Гейзенберга. Для этого в разделах 2--9 исследуется семейство римановых задач $P_{\eps}$, $\eps > 0$, на группе Гейзенберга:  методами геометрической теории управления строится оптимальный синтез, описаны первая каустика, множество разреза, радиус инъективности, сферы, а также исследуется гладкость римановых сфер и расстояния. В разделе 10 рассматривается субриманова задача $P_0$ на группе Гейзенберга  и исследуется сходимость существенных объектов задач $P_{\eps}$ к соответствующим объектам задачи $P_0$ при $\eps \to 0$.

\section{Постановка задачи}
Пусть $M = \R^3_{x, y, z}$ есть группа Гейзенберга с законом умножения 
\begin{align*}
\begin{pmatrix}
x_1\\
y_1\\
z_1
\end{pmatrix} \cdot
\begin{pmatrix}
x_2\\
y_2\\
z_2
\end{pmatrix} = 
\begin{pmatrix}
x_1 + x_2\\
y_1 + y_2\\
z_1 + z_2 + (x_1 y_2 - x_2 y_1)/2
\end{pmatrix}.
\end{align*}
Векторные поля
$$
X_1 = \pder{}{x} - \frac{y}{2}\pder{}{z}, \qquad X_2 = \pder{}{y} + \frac{x}{2}\pder{}{z}, \qquad X_3 = \pder{}{z} 
$$
образуют левоинвариантный репер на $M$.
 
Пусть $\eps >0$  --- параметр, и
рассмотрим следующее семейство задач оптимального управления
\begin{align}
&\dot q = u_1 X_1 + u_2 X_2 + \eps u_3 X_3, \qquad q \in M, \label{pr31}\\
&u \in U = \{(u_1, u_2, u_3) \in \R^3 \mid u_1^2+u_2^2+u_3^2 \leq 1\},  \label{pr32}\\
&q(0) = q_0 = \Id = (0, 0, 0), \qquad q(t_1) = q_1, \\
&t_1 \to \min.  \label{pr34}
\end{align}
Решения этой задачи --- римановы кратчайшие для римановой задачи на $M$  с ортонормированным репером $\{X_1, X_2, \eps X_3\}$.
При $\eps \to 0$  система \eq{pr31} стремится к системе $\dot q = u_1 X_1 + u_2 X_2$,  поэтому естественно ожидать, что решения задачи \eq{pr31}--\eq{pr34}  стремятся к решениям субримановой задачи на группе Гейзенберга \cite{ABB, umn1}. 

Система \eq{pr31}  глобально управляема. Эта система левоинвариантна, поэтому в задаче \eq{pr31}--\eq{pr34} существуют оптимальные траектории \cite{ABB}.

\section{Экстремали}
Так как задача \eq{pr31}--\eq{pr34}  риманова, анормальных траекторий нет.

Натурально параметризованные экстремали --- траектории гамильтоновой системы с гамильтонианом $H = \frac 12 (h_1^2+h_2^2+\eps^2h_3^2)$, лежащие на поверхности уровня $\{H = \frac 12\}$, где $h_i(\lam) = \langle \lam, X_i\rangle$, $i = 1, 2, 3$, $\lam \in T^*M$:
\begin{align*}
&\dot h_1 = - h_2 h_3, \qquad \dot h_2 = h_1h_3, \qquad \dot h_3 = 0, \\
&\dot q = h_1 X_1 + h_2 X_2 + \eps^2 h_3 X_3, \\
&h_1^2 + h_2^2 + \eps^2 h_3^2 \equiv 1. 
\end{align*}

Если $h_3 = 0$,  то $h_1, h_2 \equiv \const$  и
$$
x = h_1 t, \qquad y = h_2 t, \qquad  z = 0.
$$

Если $h_3 \neq 0$, то
\begin{align}
&h_1 = h_1^0 \cos \tau - h_2^0 \sin \tau, \qquad h_2 = h_2^0 \cos \tau + h_1^0 \sin \tau, \qquad \tau = h_3 t, \nonumber\\
&x = (h_1^0 \sin \tau + h_2^0(\cos \tau - 1))/h_3, \label{xt}\\
&y = (h_2^0 \sin \tau - h_1^0(\cos \tau - 1))/h_3, \label{yt}\\
&z = (1 - \eps^2 h_3^2)(\tau - \sin \tau)/(2h_3^2) + \eps^2 \tau. \label{zt}
\end{align}

Введем параметризацию эллипсоида $\{H = \frac 12\}$:
\be{par}
h_1 = \cos \theta \cos \f, \quad h_2 = \cos \theta \sin \f, \quad \eps h_3 = \sin \theta, \quad \theta \in \left[\frac{\pi}{2}, \frac{\pi}{2}\right], \ \  \f \in R/(2 \pi \Z).
\ee
Экстремальные траектории параметризуют экспоненциальное отображение
$$
\map{\Exp}{(T_{q_0}^*M \cap \{H = 1/2\})\times \R_{+t}}{M}, \qquad
(\theta, \f, t) \mapsto (x, y, z).
$$

\section{Сопряженные точки}
Пусть $\theta \in (-\frac{\pi}{2}, 0)\cup(0, \frac{\pi}{2})$ и  $\f \in \R/(2 \pi \Z)$.
Непосредственное вычисление на основе формул \eq{xt}--\eq{par}  дает якобиан
\begin{align*}
&J = \frac{\partial(x, y, z)}{\partial (\theta, \f, t)} = \frac{\eps^4 \cos \theta}{\sin^4 \theta} f(\theta, \tau), \\
&f(\theta, \tau) = 2(\cos \tau - 1) + \tau \sin \tau \cos^2 \theta.
\end{align*}
Легко видеть, что $2(\cos \tau -1) + \tau \sin \tau <0$  при $\tau \in (0, 2 \pi)$. Поэтому $f(\theta, \tau) < 0$  при $\tau \in (0, 2 \pi)$, $\theta \in [-\frac{\pi}{2}, \frac{\pi}{2}]$.  Более того, $f(\theta, 2 \pi) = 0$.  Доказано следующее 

\begin{proposition}\label{prop:conj}
Пусть $\theta \in (-\frac{\pi}{2}, 0)\cup(0, \frac{\pi}{2})$ и  $\f \in \R/(2 \pi \Z)$. Тогда первое сопряженное время для траектории $\Exp(\theta, \f)$ есть $\frac{2 \pi \eps}{|\sin \theta|}$.  Соответствующая сопряженная точка есть $(x, y, z) = \left(0, 0, \pi \eps^2 \frac{1 + \sin^2 \theta}{\sin^2 \theta}\right)$.
\end{proposition}

\section[Диффеоморфное свойство экспоненциального отображения]{Диффеоморфное свойство \\экспоненциального отображения}
Обозначим области в прообразе и образе экспоненциального отображения:
\begin{align*}
&N = \{(\theta, \f, \tau) \mid \theta \in (0, {\pi}/{2}), \ \f \in \R/(2 \pi \Z), \ \tau \in (0, 2 \pi)\}, \\
&D = \{(x, y, z) \in M \mid x^2+y^2 > 0, \ z > 0\}.
\end{align*} 

\begin{proposition}\label{prop:diff}
\begin{itemize}
\item[$(1)$]
$\Exp(N) \subset D$.
\item[$(2)$]
Отображение $\map{\Exp}{N}{D}$  есть диффеоморфизм.
\end{itemize}
\end{proposition}
\begin{proof}
Пусть $(\theta, \f, \tau)\in N$, и $q = (x, y, z) = \Exp(\theta, \f, \tau)$.

(1) Из формулы \eq{zt} следует, что $z > 0$. Непосредственное вычисление на основе формул \eq{xt}, \eq{yt}  дает
$$
x^2 + y^2 = \frac{2}{h_3^2}((h_1^0)^2 + (h_2^0)^2)(1 - \cos \tau) > 0.
$$
Поэтому $ q \in D$.

(2) Многообразия $N$, $D$  гомеоморфны пространству $S^1 \times \R^2$. 

Отображение $\map{\Exp}{N}{D}$  невырождено по предложению \ref{prop:conj}.

Докажем, что это отображение собственное, т.е. для любого компакта $K \subset D$ его прообраз $\Exp^{-1}(K)\subset N$ компактен. Это эквивалентно следующей импликации:
\be{implic}
(\theta_n, \f_n, \tau_n) \to \partial N \then q_n = \Exp(\theta_n, \f_n, \tau_n) \to \partial D.
\ee
Заметим, что последовательность $(\theta, \f, \tau) \to \partial N$  тогда и только тогда, когда существует подпоследовательность, на которой выполняется одно из следующих условий: 1) $\theta \to 0$, 2) $\theta \to \pi/2$, 3) $\tau \to 0$, 4) $\tau \to 2 \pi$, а последовательность $q = (x, y, z) \to \partial D$ тогда и только тогда, когда существует подпоследовательность, на которой выполняется одно из следующих условий: 1) $x^2 + y^2 \to 0$, 2) $x^2 + y^2 \to \infty$, 3) $z \to 0$, 4) $z \to \infty$. 

1) Пусть $\theta \to 0$.

1.1) Пусть $\tau \to \bar \tau \in (0, 2 \pi)$,  тогда $x^2 + y^2 \to 0$ в силу \eq{xt}, \eq{yt}.

1.2)  Пусть $\tau \to 0$.

1.2.1) Пусть $\frac{\tau}{h_3} \to 0$,  тогда $x^2 + y^2 \to 0$ в силу \eq{xt}, \eq{yt}.

1.2.2) Пусть $\frac{\tau}{h_3} \to \infty$,  тогда $x^2 + y^2 \to 0$ в силу \eq{xt}, \eq{yt}.

1.2.3) Пусть $\frac{\tau}{h_3} \to C \in (0, \infty)$,  тогда $z \to 0$ в силу \eq{zt}.

1.3) Пусть $\tau \to 2 \pi$.

1.3.1) Пусть $\frac{\tau - 2 \pi}{h_3} \to 0$, тогда $x^2 + y^2 \to 0$ в силу \eq{xt}, \eq{yt}.

1.3.2) Пусть $\frac{\tau - 2 \pi}{h_3} \to \infty$, тогда $x^2 + y^2 \to 0$ в силу \eq{xt}, \eq{yt}.

1.3.3) Пусть $\frac{\tau - 2 \pi}{h_3} \to C \in (0, \infty)$, тогда $z \to 0$ в силу \eq{zt}.

2) Пусть $\theta \to \frac{\pi}{2}$, тогда $x^2 + y^2 \to 0$ в силу \eq{xt}, \eq{yt}.

3) Пусть $\tau \to 0$, $\theta \to \bar{\theta} \in (0, \frac{\pi}{2})$, тогда $z \to 0$ в силу \eq{zt}.

Доказана импликация \eq{implic}, поэтому отображение $\map{\Exp}{N}{D}$  собственное.

По теореме 13.23 \cite{ABB}  это отображение --- накрытие.

Теперь докажем, что существует такая петля $\gamma \subset N$, что кривая $\Exp(\gamma)$ гомотопна кривой $S^1 \times \{p\}$  для некоторой точки $p \in E$, где $D$ гомеоморфно пространству $S^1 \times E$.
А именно, рассмотрим кривую
$$
\gamma \ : \ \theta = \bar \theta \in (0, {\pi}/{2}), \quad \tau = \pi, \quad \f \in S^1. 
$$ 
Тогда для кривой $\Exp(\gamma) = (x, y, z)$  имеем
\begin{align*}
&x = -2\frac{\cos \bar\theta}{\bar h_3} \sin(\f),\\
&y = 2\frac{\cos \bar\theta}{\bar h_3} \cos(\f),\\
&z = \bar z = \frac{1 + \eps^2 \bar h_3^2}{2 \bar h_3^2} \pi, \qquad \bar h_3 = \frac{\sin \bar \theta}{\eps}.
\end{align*}
Легко видеть, что кривая $\Exp(\gamma)$ гомотопна кривой $S^1 \times \bar p$,  где \\
$
\bar p = \left(0, 2 \frac{\cos \bar\theta}{\bar h_3}, \bar z\right).
$

По следствию 13.25 \cite{ABB}, отображение $\map{\Exp}{N}{D}$ есть диффеоморфизм.
\end{proof}

\section{Оптимальность экстремальных траекторий}

\begin{theorem}\label{th:optim}
Пусть $\theta \in [-\frac{\pi}{2}, \frac{\pi}{2}]$, $\f \in \R/(2 \pi \Z)$, и пусть $q(t) = \Exp(\theta, \f, t)$.
\begin{itemize}
\item[$(1)$]
Если $\theta = 0$, то траектория $q(t)$, $t \in [0, t_1]$,  оптимальна для любого $t_1 > 0$.
\item[$(2)$]
Если $\theta \neq 0$, то траектория $q(t)$, $t \in [0, t_1]$,  оптимальна тогда и только тогда, когда $t_1 \in \left(0, 
\frac{2 \pi\eps}{|\sin\theta|}\right)$.
\end{itemize}
\end{theorem}
\begin{proof}
(1)  Если $\theta = 0$, то соответствующая траектория есть прямая $q(t) = (t \cos \f, t \sin \f, 0)$. Возьмем любое $t_1 > 0$  и обозначим $q_1 = q(t_1)$. Никакая траектория $\Exp(\theta, \f, t)$  с  $\theta \neq 0$  не приходит в точку $q_1$.  Поэтому $q(t)$  есть единственная натурально параметризованная экстремальная траектория, приходящая в $q_1$. В силу существования оптимальных траекторий, $q(t)$, $t \in [0, t_1]$,  оптимальна. 

(2) Пусть $\theta \notin \{0, \pm \frac{\pi}{2}\}$. Тогда любая траектория $q(t) = \Exp(\theta, \f, t)$, $\f \in \R/(2 \pi \Z)$,  приходит в точку $\bar q = \Exp\left(\theta, \f, \frac{2 \pi \eps}{|\sin \theta|}\right) = \left(0, 0, \frac{1 + \eps^2 h_3^2}{2 h_3^2} 2 \pi\right)$, $h_3 = \frac{ \sin \theta}{\eps}$,  поэтому точка $\bar q$ есть точка Максвелла \cite{umn1}. Поэтому траектория $q(t)$, $t \in [0, t_1]$,  неоптимальна  при $t_1 > \frac{2 \pi\eps}{|\sin \theta|}$.

Пусть $ q_1 = (x_1, y_1, z_1) \in M$, $x_1^2+y_1^2 > 0$, $z_1 > 0$. Если $\theta \in \{0, \pm\frac{\pi}{2}\}$,  то траектория $\Exp(\theta, \f, t)$ не приходит в точку $q_1$. Так как отображение $\map{\Exp}{N}{D}$  есть диффеоморфизм (предложение \ref{prop:diff}),  точка $q_1$  достигается в точности одной экстремальной траекторией вида $\Exp(\theta, \f, t)$, $\theta \in (0, \frac{\pi}{2})$, $\f \in \R/(2 \pi \Z)$. В силу существования оптимальных траекторий, эта траектория оптимальна. Более того, траектория $q(t)$, $t \in [0, t_1]$, оптимальна при $t_1 \in \left(0, \frac{2 \pi \eps}{|\sin \theta|}\right)$. 

Пусть $q_1 = (0, 0, z_1) \in M$, $z_1 \in (0, 2 \pi \eps^2]$. Никакая траектория $q=\Exp(\theta, \f, t)$, $\theta \in (0, \frac{\pi}{2})$, не приходит в точку $q_1$  при $t = \frac{2 \pi}{\sin \theta}$  т.к. формула~\eq{zt} при $t = \frac{2 \pi}{\sin \theta}$ дает
$$
z = \frac{1}{2 h_3^2}(1 + \eps^2 h_3^2)\tau = \pi \eps^2\left( \frac{1}{\sin^2 \theta} + 1\right) > 2 \pi \eps^2.
$$ 
Поэтому траектория $\Exp(\frac{\pi}{2}, \f, t)$ оптимальна при $t \in [0, t_1]$, $t_1 \in \left(0, \frac{2\pi}{\eps}\right)$. В силу симметрии, траектория $\Exp(-\frac{\pi}{2}, \f, t)$ также оптимальна при $t \in [0, t_1]$, $t_1 \in \left(0, \frac{2\pi}{\eps}\right)$.

Наконец, пусть $q_1 = (0, 0, z_1) \in M$, $z_1 > 2 \pi\eps^2$. Точка $q_1$  достигается траекторией $\Exp(\frac{\pi}{2}, \f, t)$ при $t = t_1 = \frac{z_1}{\eps}$, и траекториями $\Exp(\theta,\f, t)$, $\theta = \arcsin\frac{1}{\sqrt{\frac{z_1}{\pi\eps^2}-1}} \in (0, \frac{\pi}{2})$  при $t = t_2 = 2 \pi \eps  \sqrt{\frac{z_1}{\pi\eps^2}-1}$. Но $t_2 < t_1$  при $z_1 > 2 \pi \eps^2$. Поэтому траектория $\Exp(\frac{\pi}{2}, \f, t)$, $ t\in [0, t_1]$, оптимальна тогда и только тогда, когда $t_1 \leq \frac{2\pi}{\eps}$, а траектория $\Exp(\theta, \f, t)$, $\theta \in (0, \frac{\pi}{2})$, $t \in [0, t_1]$,  оптимальна тогда и только тогда, когда $t_1 \leq \frac{2\pi}{\eps \sin \theta}$. В силу симметрии, траектория $\Exp(-\frac{\pi}{2}, \f, t)$, $ t\in [0, t_1]$, оптимальна тогда и только тогда, когда $t_1 \leq \frac{2\pi}{\eps}$, а траектория $\Exp(\theta, \f, t)$, $\theta \in (-\frac{\pi}{2}, 0)$, $t \in [0, t_1]$,  оптимальна тогда и только тогда, когда $t_1 \leq \frac{2\pi}{\eps |\sin \theta|}$.
\end{proof}

\section{Первая каустика и множество разреза}
Обозначим первое сопряженное время 
$$
\tconj(\theta, \f) = \inf \{ t_1 > 0 \mid \Exp(\theta, \f, t_1) \text{ --- сопряженная точка}\}
$$
 и время разреза
$$
\tcut(\theta, \f) = \sup \{ t_1 > 0 \mid \Exp(\theta, \f, t), \ t \in [0, t_1] \text{ оптимальна}\}
$$
вдоль экстремальной траектории $\Exp(\theta, \f, t)$, а также первую каустику
$$
\Conj = \{\Exp(\theta, \f, \tconj(\theta, \f)) \mid (\theta, \f) \in T_{q_0}^*M \cap \{H = 1/2\}\}
$$
 и множество разреза
$$
\Cut = \{\Exp(\theta, \f, \tcut(\theta, \f)) \mid (\theta, \f) \in T_{q_0}^*M \cap \{H = 1/2\}\}.
$$

\begin{theorem}\label{th:conjcut}
\begin{itemize}
\item[$(1)$]
Для всех $\theta \in [-\frac{\pi}{2}, \frac{\pi}{2}]$, $\f \in \R/(2 \pi \Z)$
$$
\tconj(\theta, \f) = \tcut(\theta, \f) = \frac{2 \pi \eps}{|\sin \theta|}.
$$
\item[$(2)$]
$\Conj = \Cut = \{(0, 0, z) \in M \mid |z| \geq 2 \pi \eps^2\}$.
\end{itemize}
\end{theorem}\
\begin{proof}
(1) Равенство для $\tconj(\theta, \f)$ доказано в предложении \ref{prop:conj}. Равенство для $\tcut(\theta, \f)$ следует из п. (2) теоремы  \ref{th:optim}.

(2) следует из п. (1).
\end{proof}

\section{Гладкость расстояния и сфер}
Обозначим через $d(q)$ риманово расстояние между точками $q_0=\Id$  и $q \in M$,
и пусть
$$
S(r) = \{q \in M \mid d(q) = r\}, \qquad r > 0,
$$
 есть риманова сфера радиуса $r$  с центром в $q_0$.

\begin{proposition}
\item[$(1)$]
Расстояние $d$  является гладким в точке $q \in M$  тогда и только тогда, когда $q \notin \Cut \cup \{q_0\}$.
\item[$(2)$]
Если $r \in (0, 2 \pi \eps^2)$, то сфера $S(r)$ гладкая. 
Если $r =2 \pi \eps^2$, то сфера $S(r)$ гладкая в точках, где $x^2 + y^2 \neq 0$. 
Если $r >2 \pi \eps^2$, то сфера $S(r)$ гладкая в точках, где $x^2 + y^2 \neq 0$, и негладкая в точках, где $x^2 + y^2 = 0$.
\end{proposition}
\begin{proof}
(1) Расстояние $d$ является гладким в точке $q \in M \setminus \{q_0\}$ тогда и только тогда, когда существует единственная кратчайшая, соединяющая $q_0$  и $q$, и точка $q$  не является сопряженной \cite{ABB}.

(2) Если в точку $q \in M \setminus \{q_0\}$  приходит единственная кратчайшая и $q$  не является сопряженной точкой,  то сфера гладкая в точке $q$. Если в точку~$q$  приходит более одной кратчайшей, то сфера негладкая в точке $q$ \cite{ABB}.
\end{proof}

\begin{remark}
Гладкость сферы $S(r)$, $r = 2 \pi \eps$, в точках прямой $\{x^2 + y^2 = 0\}$  остается под вопросом. В окрестности таких точек эта сфера параметризуется переменными $h_1^0$, $h_2^0$ (причем в таких точках $h_1^0 = h_2^0 = 0$)  соотношениями вида $(h_1^0, h_2^0)  \mapsto (x, y, z)$. Непосредственное вычисление показывает, что
$$
\frac{\partial(x, y)}{\partial(h_1^0, h_2^0)} = 
\frac{\partial(x, z)}{\partial(h_1^0, h_2^0)} =
\frac{\partial(y, z)}{\partial(h_1^0, h_2^0)} = 0 \text{ при }  h_1^0 = h_2^0 = 0.
$$
\end{remark}

\section{Радиус инъективности}
Радиус инъективности риманова многообразия $M$  можно определить следующим образом \cite{burago}:
\begin{multline*}
r_i := \sup \{t_1 > 0 \mid \Exp(\theta, \f, t), \ t \in [0, t_1], \text{ --- кратчайшая}, \\ \theta \in [-{\pi}/{2}, {\pi}/{2}], \ \f \in \R/(2 \pi \Z)\}. 
\end{multline*}

Из теоремы \ref{th:conjcut}  получаем
\begin{corollary}
$r_i = 2 \pi \eps$.
\end{corollary}

\section{Сходимость при $\eps \to 0$}
В этом разделе будем обозначать риманову задачу \eq{pr31}--\eq{pr34} как $P_{\eps}$, $\eps > 0$, и рассмотрим субриманову задачу $P_0$, которая ставится следующим образом \cite{ABB, umn1}:
\begin{align}
&\dot q = u_1 X_1 + u_2 X_2, \qquad q \in M, \quad u_1^2 + u_2^2 \leq 1, \\
&q(0) = q_0, \quad q(t_1) = q_1, \\
&t_1 \to \min. 
\end{align} 

\subsection{Сходимость экспоненциального отображения}
Натурально параметризованные нормальные экстремальные траектории задачи $P_0$  задаются следующим образом:
\begin{align}
&x = (\sin(\psi + c t) - \sin \psi)/c, \label{xt4}\\
&y = (\cos \psi - \cos(\psi + c t))/c, \label{yt4}\\
&z = (ct-\sin ct)/(2 c^2) \label{zt4}
\end{align}
при $c \neq 0$  и 
\be{xyzt4}
x = t \cos \psi, \qquad y = t \sin \psi, \qquad z =0
\ee
при $c = 0$,  где $\psi \in \R/(2 \pi \Z)$, $c \in \R$. Формулы \eq{xt4}--\eq{xyzt4}  параметризуют экспоненциальное отображение для задачи $P_0$:
$$
\map{\Exp_0}{(T_{q_0}^* M \cap \{h_1^2+h_2^2 = 1\})\times \R_{+t}}{M}, \qquad (c,\psi,  t) \mapsto (x, y, z),
$$
где $h_1 = \cos \psi$, $h_2 = \sin \psi$, $h_3 = c$.

Для сравнения экспоненциальных отображений $\Exp_{\eps}$ для задачи $P_{\eps}$, $\eps > 0$,  и $\Exp_0$ для задачи $P_0$ определим следующие функции при $h_1^2 + h_2^2 \neq 0$:
\be{change}
(\theta, \f) \mapsto (c, \psi), \qquad c = \frac{\sin \theta}{\eps}, \quad \psi = \f.
\ee

\begin{proposition}\label{prop:Expe}
Пусть $\theta \in (-\frac{\pi}{2}, \frac{\pi}{2})$, $\f \in \R/(2 \pi \Z)$.
Тогда имеет место равенство
$$
\lim_{\eps \to 0} \Exp_{\eps}(\theta, \f, t) = \restr{\Exp_0(c, \psi, t)}{c = c(\theta), \ \psi = \psi(\f)}.
$$
\end{proposition}
\begin{proof}
Обозначим $\Exp_{\eps}(\theta, \f, t) = (x_{\eps}, y_{\eps}, z_{\eps})$, \\
$ \restr{\Exp_0(c, \psi, t)}{c = c(\theta), \ \psi = \psi(\f)} = (x_0, y_0, z_0)$.  Если $\theta \in (-\frac{\pi}{2}, \frac{\pi}{2})$, то
при $\eps \to 0$
\begin{align*}
&x - x_0 = 
\frac{2 \eps}{|\cos \theta|} 
\left(\cos \left(\f - \frac{t \sin \theta}{2 \eps}\right) + 
|\cos \theta| \cos\left(\f + \frac{t \sin \theta}{2 \eps}\right)\right) \ctg \theta \sin \frac{t \sin \theta}{2 \eps} \to 0,\\
&y-y_0 = 
\frac{\eps \ctg \theta}{|\cos \theta|} 
\left(\cos \f |\cos \theta| - 1\right) + (1 - |\cos \theta|)
\left(\cos\left(\f + \frac{t \sin \theta}{\eps}\right) + \sin \f\right)  \to 0,\\
&z - z_0 = \frac{\eps}{2}
\left(t \sin \theta + \eps \sin\left(\frac{t \sin \theta}{\eps}\right)\right)  \to 0. 
\end{align*}

Если $\theta = 0$,  то $x = x_0$, $y = y_0$, $z = z_0$.
\end{proof}

Некоторые оптимальные траектории в задачах $P_1$  и $P_0$  приведены соответственно на рис. \ref{fig:traj}  и \ref{fig:traj0}.

\twofiglabelsize
{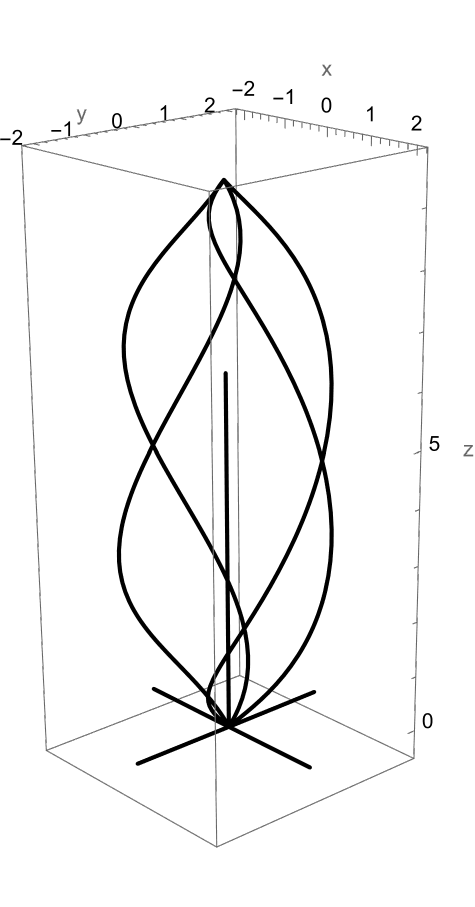}{Оптимальные траектории в задаче $P_{\eps}$, $\eps = 1$}{fig:traj}{0.8}
{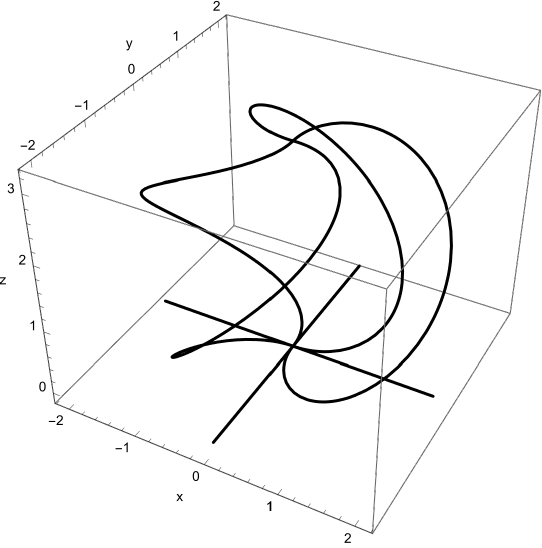}{Оптимальные траектории в задаче $P_{0}$}{fig:traj0}{0.6}

\subsection{Сходимость первой каустики и множества разреза}

Обозначим через $\Conje$  и $\Cut_{\eps}$  первую каустику и множество разреза для задачи $P_{\eps}$. Мы показали в теореме \ref{th:conjcut}, что
\be{conjeps}
\Conje = \Cut_{\eps} = \{(0, 0, z) \mid |z| \geq 2 \pi \eps^2 \}.
\ee
Для задачи $P_0$  первая каустика и множество разреза равны
\be{conj0}
\Conjo = \Cut_0 =  \{(0, 0, z) \mid |z| \neq 0 \}.
\ee

Обозначим через $\chi_A$   характеристическую функцию множества $A$. 

\begin{proposition}
\begin{itemize}
\item[$(1)$]
Если $0 \leq \eps_1 < \eps_2$,  то $ \Cut_{\eps_1} \supset \Cut_{\eps_2}$.
\item[$(2)$]
Для любого $q \in M$ $\displaystyle\lim_{\eps \to 0} \chi_{\Cut_{\eps}}(q) = \chi_{\Cut_{0}}(q)$.
\end{itemize}
\end{proposition}
\begin{proof}
(1) следует из формул \eq{conjeps}, \eq{conj0}.

(2) Пусть $q = (x, y, z) \in M$. Если $x^2 + y^2 \neq 0$, то $\chi_{\Cut_{\eps}}(q) = \chi_{\Cut_{0}}(q) = 0$. 
Если $x^2 + y^2 =z =0$, то $\chi_{\Cut_{\eps}}(q) = \chi_{\Cut_{0}}(q) = 0$.
Наконец, пусть  $x^2 + y^2 \neq 0$ и $z \neq 0$. Тогда $\chi_{\Cut_{0}}(q) = 1$. Если  $\eps \leq \sqrt{\frac{|z|}{2 \pi}}$, то $|z| \geq 2 \pi \eps^2$, поэтому  $\chi_{\Cut_{\eps}}(q) = 1$.
\end{proof}

\subsection{Сходимость сфер}
Риманова сфера радиуса $r > 0$  в задаче $P_{\eps}$, $\eps > 0$,  параметризуется следующим образом:
\be{pare}
S_{\eps}(r) = \{\Exp_{\eps}(\theta, \f, r) \mid r|\sin \theta| \leq 2 \pi \eps, \ \f \in \R/(2 \pi \Z)\},
\ee
 а субриманова сфера радиуса $r > 0$  в задаче $P_0$ --- следующим:
$$
S_{0}(r) = \{\Exp_{0}(c, \psi, r) \mid r|c| \leq 2 \pi, \ \psi \in \R/(2 \pi \Z)\},
$$
или, с учетом замены переменных \eq{change},
\be{par0}
S_{0}(r) = \left\{\Exp_{0}\left(\frac{\sin \theta}{\eps}, \f, r\right) \mid r|\sin \theta| \leq 2 \pi \eps, \ \f \in \R/(2 \pi \Z)\right\}.
\ee

\begin{proposition}
Параметризация \eq{pare}  сферы $S_{\eps}(r)$  сходится к параметризации \eq{par0}  сферы $S_0(r)$  при $\eps \to 0$.
\end{proposition}
\begin{proof}
Достаточно воспользоваться предложением \ref{prop:Expe}:
$$
\lim_{\eps \to 0} \Exp_{\eps}(\theta, \f, r) = \Exp_0\left(\frac{\sin \theta}{\eps}, \f, r\right).
$$
\end{proof}

Напомним необходимые определения о сходимости множеств \cite{aubin}. Пусть дано семейство множеств $A_{\eps}$, $\eps \geq 0$,  в метрическом пространстве $(M, \rho)$. Нижний и верхний пределы семейства $A_{\eps}$  при $\eps \to 0$  определяются соответственно как
$$
\liminf_{\eps \to 0} A_{\eps} := \{ q \in M \mid \lim_{\eps \to 0} \rho(q, A_{\eps}) = 0\}
$$
и
$$
\limsup_{\eps \to 0} A_{\eps} := \{ q \in M \mid \liminf_{\eps \to 0} \rho(q, A_{\eps}) = 0\},
$$
где $\rho(q, A) := \inf_{p \in A} \rho(q, p)$. Семейство $A_{\eps}$ называется полунепрерывным снизу (сверху)  при $\eps \to 0$  по Куратовскому, если $A_0 \subset \liminf_{\eps\to 0} A_{\eps}$ (соотв. $A_0 \supset \limsup_{\eps\to 0} A_{\eps}$). Наконец, семейство множеств непрерывно, если оно полунепрерывно снизу и сверху.

\begin{proposition}
Для любого $r > 0$ семейство сфер $S_{\eps}(r)$  полунепрерывно снизу при $\eps \to 0$ по Куратовскому.
\end{proposition}
\begin{proof}
Пусть $q \in S_0(r)$, тогда $q = \Exp(c, \psi, r)$  для некоторых $c \in \R$, $\psi \in \R/(2 \pi \Z)$. Положим $q_{\eps} = \Exp_{\eps}(\theta(c), \f(\psi), r) \in S_{\eps}(r)$. По предложению~\ref{prop:Expe}  имеем $\lim_{\eps\to 0} q_{\eps} = q_0$,  поэтому $S_{0}(r) \subset \liminf_{\eps\to 0} S_{\eps}(r)$.
\end{proof}

Сферы в задачах $P_{\eps}$, $\eps \geq 0$,  являются поверхностями вращения вокруг оси $z$.
Сечения плоскостью $\{y = 0\}$
субримановых сфер в задаче $P_0$  радиусов $2, 1, 0,\!5, 0,\!25$  изображены на рис. \ref{fig:spheps0.pdf}. 
Сечения плоскостью $\{y = 0\}$
римановых сфер в задаче $P_{\eps}$, $\eps = 0,\!1$  радиусов $2, 1, 0,\!5, 0,\!25$  изображены на рис. \ref{fig:spheps01.pdf}. 
Сечения плоскостью $\{y = 0\}$
римановых сфер в задаче $P_{\eps}$, $\eps = 1$  радиусов $2, 1, 0,\!5, 0,\!25$  изображены на рис. \ref{fig:spheps1.pdf}.
Также на рис. \ref{fig:spheps0.pdf}--\ref{fig:spheps1.pdf} изображено множество разреза в соответствующей задаче.

На рис. \ref{fig:spheps1r1.pdf} и \ref{fig:spheps1r2pi.pdf} изображены римановы сферы радиуса 1 и $2 \pi$ в задаче  $P_{\eps}$, $\eps = 1$.

Наконец, на рис. \ref{fig:spheps0r01.pdf} и \ref{fig:spheps01r01.pdf} изображены сферы радиуса $0,\!1$  в задачах $P_0$ и $P_{\eps}$, $\eps = 0,\!1$.

Рисунки \ref{fig:spheps0.pdf},  \ref{fig:spheps01.pdf}, \ref{fig:spheps0r01.pdf}, \ref{fig:spheps01r01.pdf}  позволяют выдвинуть гипотезу, что семейство сфер $S_{\eps}(r)$  непрерывно по Куратовскому при $\eps \to 0$.

\twofiglabelsize
{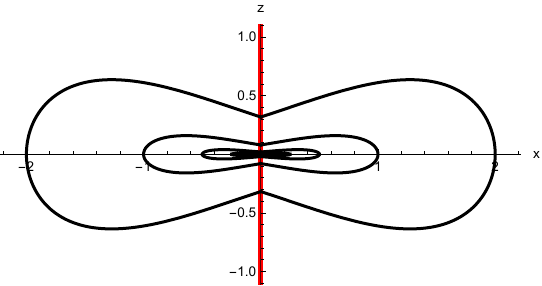}{Субримановы сферы в задаче $P_0$}{fig:spheps0.pdf}{0.3}
{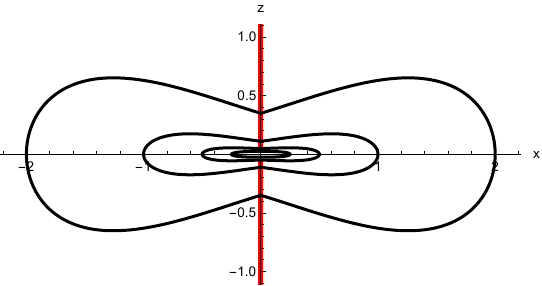}{Римановы сферы в задаче $P_{\eps}$, $\eps = 0,\!1$}{fig:spheps01.pdf}{0.3}

\onefiglabelsizen
{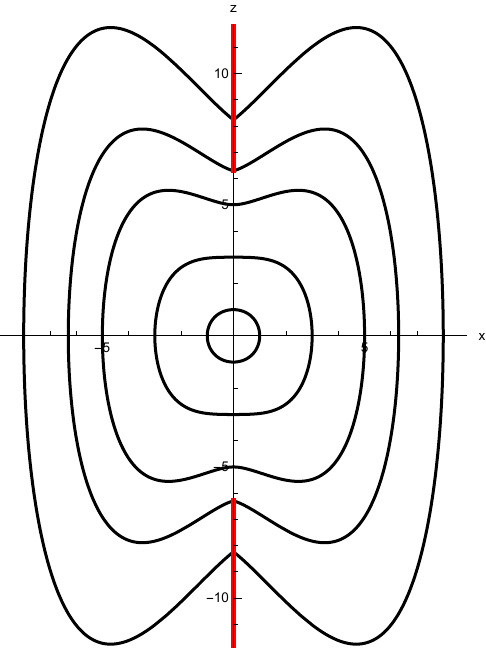}{Римановы сферы в задаче $P_{\eps}$, $\eps = 1$}{fig:spheps1.pdf}{8}

\twofiglabelsize
{spheps1r1.pdf}{Риманова сфера радиуса 1 в задаче $P_{\eps}$, $\eps = 1$}{fig:spheps1r1.pdf}{0.6}
{spheps1r2pi.pdf}{Риманова сфера радиуса $2\pi$ в задаче $P_{\eps}$, $\eps = 1$}{fig:spheps1r2pi.pdf}{0.6}

\twofiglabelsize
{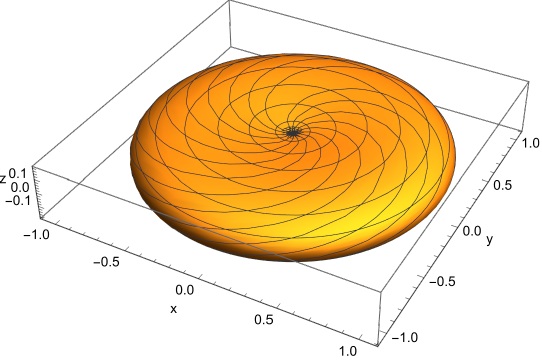}{Субриманова сфера радиуса $0,\!1$ в задаче $P_{0}$}{fig:spheps0r01.pdf}{0.3}
{spheps01R1.pdf}{Риманова сфера радиуса $0,\!1$ в задаче $P_{\eps}$, $\eps = 0,\!1$}{fig:spheps01r01.pdf}{0.3}

\end{document}